\documentclass[a4paper,reqno,12pt]{amsart}
\usepackage{verbatim}
\usepackage{amssymb,amsmath,amsthm}
\usepackage{amsfonts}
\usepackage{array}
\usepackage{xcolor}

\def\op{\overline{p}}

\newcommand{\Uodd}{\widetilde{U}}
\newcommand{\modd}{\mathfrak{m}_{{\text {\rm odd}}}}

\newtheorem{theorem}{Theorem}[section]
\newtheorem{lemma}{Lemma}[section]
\newtheorem{corollary}{Corollary}[section]

\newtheoremstyle{remark}
    {\dimexpr\topsep/2\relax} 
    {\dimexpr\topsep/2\relax} 
    {}          
    {}          
    {\bfseries} 
    {.}         
    {.5em}      
    {}          

\theoremstyle{remark}

\makeatletter
\@namedef{subjclassname@2020}{%
  \textup{2020} Mathematics Subject Classification}
\makeatother

\begin{document}

\title[]
{Arithmetic properties of \\MacMahon-type sums of divisors:\\
the odd case}

\author{James A. Sellers}
\address{Department of Mathematics and Statistics, 
University of Minnesota Duluth, Duluth, MN 55812, USA}
\email{jsellers@d.umn.edu}

\author{Roberto Tauraso}
\address{Dipartimento di Matematica, 
Università di Roma ``Tor Vergata'', 00133 Roma, Italy}
\email{tauraso@mat.uniroma2.it}

\subjclass[2020]{Primary 11P83; Secondary 11P81, 05A17.}

\keywords{partitions, congruences.}

\begin{abstract} 
A century ago, P. A.~MacMahon introduced two families of generating functions, 
$$
\sum_{1\leq n_1<n_2<\cdots<n_t}\prod_{k=1}^t\frac{q^{n_k}}{(1-q^{n_k})^2}
\quad\text{ and }
\sum_{\substack{1\leq n_1<n_2<\cdots<n_t\\ \text{$n_1,n_2,\dots,n_t$ odd}}}\prod_{k=1}^t\frac{q^{n_k}}{(1-q^{n_k})^2},
$$
which connect sum-of-divisors functions and integer partitions. 
These have recently drawn renewed attention. In particular, Amdeberhan, Andrews, and Tauraso extended the first family above by defining
$$
U_t(a,q):=\sum_{1\leq n_1<n_2<\cdots<n_t}\prod_{k=1}^t\frac{q^{n_k}}{1+aq^{n_k}+q^{2n_k}}
$$
for $a=0, \pm1, \pm2$ and investigated various properties, including some congruences satisfied by the coefficients of the power series representations for $U_t(a,q)$. These arithmetic aspects were subsequently expanded upon by the authors of the present work.  
Our goal here is to generalize the second family of generating functions, where the sums run over odd integers, and then apply similar techniques to show new infinite families of Ramanujan--like congruences for the associated power series coefficients. 
\end{abstract} 

\maketitle

\section{Introduction}

The two families of generating functions
\begin{equation}\label{Macgf}
\sum_{1\leq n_1<n_2<\cdots<n_t}\prod_{k=1}^t\frac{q^{n_k}}{(1-q^{n_k})^2}
\quad\text{ and }
\sum_{\substack{1\leq n_1<n_2<\cdots<n_t\\ \text{$n_1,n_2,\dots,n_t$ odd}}}\prod_{k=1}^t\frac{q^{n_k}}{(1-q^{n_k})^2},
\end{equation}
introduced by  P. A.~MacMahon in \cite{MacMahon}, relate integer partitions and sum-of-divisors functions.  Note that  the coefficient of $q^n$ of the first function in \eqref{Macgf} is the sum of products of the multiplicities $m_1,m_2,\dots,m_t$ such that $n$ can be partitioned as
$$n=m_1\cdot n_1+m_2\cdot n_2+\dots +m_t\cdot n_t$$
with $1\leq n_1<n_2<\cdots<n_t$. Restricting to the case of odd multiplicities, we find the coefficient of $q^n$ of the second function in \eqref{Macgf}. 

Such functions have attracted renewed interest as demonstrated by the considerable number of recent articles \cite{AAT1,ABS25,AOS24,AndrewsRose13, Bachmann24,BCIP24,Merca25, OnoSingh24,Rose15}. Furthermore, in \cite{AAT2}, Amdeberhan, Andrews, and Tauraso extended the first family in \eqref{Macgf} by defining
$$
U_t(a,q):=\sum_{1\leq n_1<n_2<\cdots<n_t}\prod_{k=1}^t\frac{q^{n_k}}{1+aq^{n_k}+q^{2n_k}}
$$
for $a=0, \pm1, \pm2$ (so that the roots of $1+ax+x^2=0$ are roots of unity). They investigated various properties including several congruences for their coefficients. These arithmetic aspects were later expanded upon in \cite{SellersTauraso25}.

In this work, we take a similar approach for the second family in \eqref{Macgf} by defining the generating functions
\begin{equation} \label{defofV} 
\Uodd_t(a,q):=\sum_{\substack{1\leq n_1<n_2<\cdots<n_t\\ \text{$n_1,n_2,\dots,n_t$ odd}}
}\prod_{k=1}^t\frac{q^{n_k}}{1+aq^{n_k}+q^{2n_k}}=\sum_{n\geq0} \modd(a,t;n)q^n,
\end{equation}
where once again $a=0, \pm1, \pm2$.  Since $\Uodd_t$ satisfies the identity
$$\Uodd_t(-a,q)=(-1)^t\Uodd_t(a,-q),$$
we will limit ourselves to the values $-2,0,1$.

When $t=1$, the functions $\Uodd_1(-2,q)$, $\Uodd_1(0,q)$ and $\Uodd_1(1,q)$ exhibit some known and easily identifiable features. For example, 
\begin{align*}
\Uodd_1(-2,q)&=\sum_{n\geq 1}\frac{q^{2n-1}}{(1-q^{2n-1})^2}=\sum_{n,m\geq 1}mq^{m(2n-1)}
=\sum_{n\geq 1}\Big(\sum_{\substack{d\mid n,\text{$n/d$ is odd}}}d\Big) q^n\\
&=q+2q^2+4 q^3+4q^4+6q^5+8q^6+8 q^7+8 q^8+13 q^9+\dots
\end{align*}
or equivalently, 
$$
\Uodd_1(-2,q)=\sum_{n\geq 1}\frac{q^{n}}{(1-q^{n})^2}
-\sum_{n\geq 1}\frac{q^{2n}}{(1-q^{2n})^2}\\
=\sum_{n\geq 1}\sigma(n)(q^n-q^{2n})
$$
which implies
$$\modd(-2,1;n)=\begin{cases}
\sigma(n)-\sigma(n/2)&\text{if $n$ is even,}\\
\sigma(n)&\text{if $n$ is odd,}
\end{cases}$$
where $\sigma(n)$ is the sum of the positive divisors of $n$.  See \cite[A002131]{OEIS}.  

\noindent  For $a=0$, we have
\begin{align*}
\Uodd_1(0,q)&=\sum_{n\geq 1}\frac{q^{2n-1}}{1+q^{2(2n-1)}}=
\sum_{n\geq 1} \frac{q^n(1 - q^n)(1 - q^{2n}) (1 - q^{3n})}{1 - q^{8n}}\\
&=q\frac{f_8^4}{f_4^2}= q+2q^5+q^9+2q^{13}+2q^{17}+3q^{25}+2q^{29}+2q^{37}+\dots\end{align*}
where $f_r=(q^r;q^r)_{\infty} = \prod_{k\geq 1}(1-q^{rk})$. 
Note that $\modd(0,1;4n+1)$ is the number of ways of writing $n$ as the sum of two triangular numbers.  
$$\modd(0,1;4n+1)=[q^n]\frac{f_2^4}{f_1^2}
=[q^n]\Big(\sum_{n\geq 0} q^{T_n}\Big)^2$$
where $T_n=n(n+1)/2$ is the $n$th triangular number.  See \cite[A008441]{OEIS}.   
 Moreover, it can be verified that, for all $n\geq 0$, 
$$\modd(0,1;4(9n+5)+1)=\modd(0,1;4(9n+8)+1)= 0.$$

\noindent On the other hand, for $a=1$, we have
\begin{align*}
\Uodd_1(1,q)&=\sum_{n\geq 1}\frac{q^{2n-1}}{(1+q^{2n-1}+q^{2(2n-1)})}
=\sum_{n\geq 1}\frac{q^{2n-1}(1-q^{2n-1})}{(1-q^{3(2n-1)})}\\
&=\sum_{n\geq 1}\frac{q^n(1-q^n)}{1-q^{3n}}
-\sum_{n\geq 1}\frac{q^{2n}(1-q^{2n})}{1-q^{6n}}\\
&=\sum_{n\geq 1}\frac{q^n-2q^{2n}+2q^{4n}-q^{5n}}{1-q^{6n}}\\
&=q - q^2 + q^3 + q^4 - q^6 + 2q^7 - q^8 + q^9 + \dots .
\end{align*}
that is
$$\modd(1,1;n)=\tau_1(n)-2\tau_2(n)+2\tau_4(n)-\tau_5(n)$$
where $\tau_j(n)$ is the number of positive divisors $d$ of $n$ such that $d\equiv j \pmod{6}$.  See \cite[A093829]{OEIS}.
Notice that 
\begin{equation}
\label{modd1_1_6n5_equal_0}
\modd(1,1;6n+5)=\tau_1(6n+5)-\tau_5(6n+5)=0
\end{equation}
because if $d$ divides $6n+5$ then $d\equiv 1 \pmod6$ if and only if $n/d\equiv 5 \pmod6$. 

With the above in mind, we see that the generating functions $\Uodd_t(a,q)$ naturally generalize various well--studied arithmetic functions.  For $t>1$, our first goal is to obtain a more convenient formula for $\Uodd_t(a,q)$ that avoids multiple summations. This will be discussed in Section \ref{explicitS}, where we find that $\Uodd_t(a,q)$ can be written as 
\begin{equation}\label{productgf}
\sum_{k\geq 0}b_k(a) q^k \cdot \sum_{n\geq 0}c_n(a,t)q^{r(n)}
\end{equation}
where $r(n)$ is a quadratic polynomial.  (Similar work was completed in \cite{SellersTauraso25} for the related functions $U_t(a,q)$.)  

After laying this groundwork, we turn our attention to congruences.
Some of these congruences are straightforward.  For example, thanks to \eqref{defofV}, we immediately see that, for any $t$,
\begin{align}
\Uodd_t(-2, q) &\equiv \Uodd_t(0, q) \pmod{2}\nonumber,\\
\label{easy3}
\Uodd_t(-2, q) &\equiv \Uodd_t(1, q) \pmod{3}.
\end{align}
In order to address less trivial congruences, we take advantage of the product in \eqref{productgf} and study the arithmetic properties of the coefficients $b_k(a)$ and $c_n(a,t)$.
From there, starting in Section \ref{am2S} and beyond, we will derive our results for $\modd(a,t;n)$, grouped according to the value of $a$.

It is worth pointing out that, for $a = 1$, we will come across the prefactor
$$\frac{f_1 f_6}{f_2^2 f_3}=1-q+q^2-q^4+2q^5-3q^6+4q^7-5q^8+7q^9+\dots $$
which appears to be novel in the literature and exhibits intriguing arithmetic properties which are of independent interest.  These will be discussed in Section \ref{pre2S}.

\section{Explicit form of $\Uodd_t(a,q)$}\label{explicitS}

In \cite[Corollary 2]{AndrewsRose13}, we have an explicit form for the case $a=-2$, namely
\begin{equation}\label{exformm2}
\Uodd_t(-2,q) = \frac{f_2}{f_1^2}\cdot
\sum_{n=1}^\infty c(-2,t)q^{n^2}
\end{equation}
where 
$$c(-2,t)=(-1)^{n+t} \frac{2n}{n+t}\binom{n+t}{2t}$$
and the prefactor is the generating function of number of overpartitions of $n$:
$$\frac{f_2}{f_1^2}=\frac{(-q;q)_\infty}{(q;q)_\infty}
=\sum_{n=0}^{\infty}\op(n)q^n=
1+2 q+4 q^2+8 q^3+14 q^4+24 q^5+\dots$$
An overpartition of $n$ is a non-increasing sequence of natural numbers whose sum is $n$ in which the first occurrence of a number may be overlined.  

The formula \eqref{exformm2} was derived by making use of the properties of the Chebychev polynomials of the first kind, defined by
$T_n(\cos\theta)=\cos(n\theta)$,
and of their representation in terms of the sum
\begin{equation}\label{ChebychevP}
T_n(x)=\frac{n}{2}\sum_{k=0}^{\lfloor n/2\rfloor}
\frac{(-1)^k}{n-k}\binom{n-k}{k}(2x)^{n-k}.
\end{equation}
As we will see below, the same method is effective for $a=0$ and $a=1$ as well.
Noting that the Chebychev polynomial $T_n(x)$ is alternately an even and odd function, depending on $n$, we define
$$to_n(x)=\frac{T_{2n+1}(\sqrt{x})}{\sqrt{x}}
\quad\text{and}\quad
te_n(x)=T_{2n}(\sqrt{x}).$$
Then, by \cite[Theorem 1]{AndrewsRose13},
$$
\sum_{n\geq 0}to_n\left(\frac{x}{4}\right)q^{\binom{n+1}{2}}
=(q;q)^3_{\infty}\prod_{n=1}^{\infty}\left(1+\frac{xq^n}{(1-q^n)^2}\right)
$$
and
\begin{equation}\label{teven}
1 + 2\sum_{n>0}te_n\left(\frac{x}{4}\right)q^{n^2}=\frac{(q;q)_\infty}{(-q;q)_\infty} \prod_{n=1}^\infty\Big(1 + \frac{xq^{2n-1}}{(1 - q^{2n-1})^2}\Big).
\end{equation}
In \cite[Theorem 2.1]{AAT2}, the authors established
$$
\sum_{t\geq0} U_t(a,q)x^t
=\prod_{n=1}^{\infty}\frac{1}{(1+aq^n+q^{2n})(1-q^{n})}
\cdot \sum_{n\geq 0}to_n\left(\frac{x+a+2}{4}\right)q^{\binom{n+1}{2}}.
$$
Along the same lines we show the following identity.
\begin{theorem} We have
\begin{align}\label{id2} 
\sum_{t\geq0} \Uodd_t(a,q)x^t
&=\prod_{n=1}^{\infty}\frac{1}{(1+aq^{2n-1}+q^{2(2n-1)})(1-q^{2n})}\nonumber\\
&\qquad\qquad\cdot \left(1+2\sum_{n\geq 1}te_n\left(\frac{x+a+2}{4}\right)q^{n^2}\right).
\end{align}
\end{theorem}
\begin{proof} We have that
\begin{align*} \sum_{t\geq0} \Uodd_t(a,q)x^t
&=\prod_{n=1}^{\infty}\left(1+\frac{xq^{2n-1}}{1+aq^{2n-1}+q^{2(2n-1)}}\right) \\
&=\prod_{n=1}^{\infty}\frac{1}{1+aq^{2n-1}+q^{2(2n-1)}}
\cdot \prod_{n=1}^{\infty}\left(1+(x+a)q^{2n-1}+q^{2(2n-1)}\right) \\
&=\prod_{n=1}^{\infty}\frac{(1-q^{2n-1})^2}{1+aq^{2n-1}+q^{2(2n-1)}}
\cdot \prod_{n=1}^{\infty}\left(1+\frac{(x+a+2)q^{2n-1}}{(1-q^{2n-1})^2}\right) \\
&=\prod_{n=1}^{\infty}\frac{1}{(1+aq^{2n-1}+q^{2(2n-1)})(1-q^{2n})}\nonumber\\
&\qquad\qquad\cdot \left(1+2\sum_{n\geq 1}te_n\left(\frac{x+a+2}{4}\right)q^{n^2}\right) 
\end{align*}
where at the last step we applied \eqref{teven}.
\end{proof}

We are now in a position to present explicit formulas for $a=0$ and $a = 1$.

\noindent By \eqref{ChebychevP}, we obtain
\begin{align}\label{tesum}
[x^t]\,2 te_n\left(\frac{x+a+2}{4}\right)
&=[x^t]\,2n\sum_{k=0}^{n}
\frac{(-1)^{n-k}}{n+k}\binom{n+k}{2k}(x+a+2)^{k}\nonumber \\
&=\sum_{k=t}^{n}(-1)^{n-k}
\frac{2n}{n+k}\binom{n+k}{2k}\binom{k}{t}(a+2)^{k-t}.
\end{align}
Moreover, the binomial sum on the right-hand side of \eqref{tesum} admits the following useful representation, which is derived by using Riordan arrays (see \cite{Sprugnoli94} and also \cite[Appendix 2]{AAT2}).

\begin{lemma} For $n\geq t\geq 1$, we have that
\begin{equation}\label{riordan}
\sum_{k=t}^{n}(-1)^{n-k}
\frac{2n}{n+k}\binom{n+k}{2k}\binom{k}{t}(a+2)^{k-t}
=[z^n]\frac{z^t-z^{t+2}}{(1-az+z^2)^{t+1}}.
\end{equation}
\end{lemma}
\begin{proof} If $F(z)= \sum_{k\geq 0}c_k z^k$ then
$$\sum_{k\geq 0}(-1)^{n-k}\frac{2n}{n+k}\binom{n+k}{2k}c_k
=[z^n]\,\frac{1-z}{1+z}\cdot
F\left(\frac{z}{(1+z)^2}\right).$$
Hence, by letting
$$F(z)=\sum_{k\geq t}\binom{k}{t}(a+2)^{k-t}z^k
=\frac{z^t}{(1-(a+2)z)^{t+1}},$$
we find that the left-hand side of \eqref{riordan} is equal to
\begin{align*}
[z^n]\,\frac{1-z}{1+z}\cdot
F\left(\frac{z}{(1+z)^2}\right)
&=[z^n]\,\frac{1-z}{1+z}\cdot
\frac{\Big(\frac{z}{(1+z)^2}\Big)^t}{\Big(1-(a+2)\frac{z}{(1+z)^2}\Big)^{t+1}}\\
&=[z^n]
\frac{z^t-z^{t+2}}{(1-az+z^2)^{t+1}}.
\end{align*}
\end{proof}

\noindent Letting $a=0$ in \eqref{riordan}, we find
$$
2n\sum_{k=t}^{n}
\frac{(-1)^{n-k}}{n+k}\binom{n+k}{2k}\binom{k}{t}2^{k-t}
=\begin{cases}
(-1)^{\frac{n-t}{2}}\frac{2n}{n+t}\binom{\frac{n+t}{2}}{t}&\text{if $n+t$ is even,}\\
0&\text{if $n+t$ is odd.}
\end{cases}
$$
Hence, from \eqref{id2}, by separately considering the even and odd cases, we obtain
\begin{equation}\label{exform0even}
\Uodd_{2t}(0,q) = \frac{f_8}{f_4^2}\cdot 
\sum_{n=1}^\infty (-1)^{n+t} \frac{2n}{n+t}\binom{n+t}{2t}q^{4n^2}=\Uodd_t(-2,q^4)
\end{equation}
and
\begin{equation}\label{exform0odd}
\Uodd_{2t+1}(0,q) = \frac{f_8}{f_4^2}\cdot
\sum_{n=1}^\infty (-1)^{n-t-1} \frac{2n-1}{n+t}\binom{n+t}{2t+1}q^{(2n-1)^2}=qW_t(q^4)
\end{equation}
with
\begin{equation}\label{exformW}
W_t(q)= \frac{f_2}{f_1^2}\cdot
\sum_{n=1}^\infty c_n(0,t)q^{n(n-1)}
\end{equation}
and
$$c_n(0,t)=(-1)^{n-t-1} \frac{2n-1}{n+t}\binom{n+t}{2t+1}.$$
Finally, for $a=1$, from \eqref{id2} and \eqref{tesum}, we get
\begin{equation}\label{exform1}
\Uodd_{t}(1,q)=\frac{f_1 f_6}{f_2^2 f_3}
\cdot \sum_{n=1}^{\infty}c_n(1,t)q^{n^2}
\end{equation}
where 
$$c_n(1,t)=\sum_{k=t}^n(-1)^{n-k}\frac{2n}{n+k}\binom{n+k}{2k}
\binom{k}{t}3^{k-t}.$$
For $a=1$, equation \eqref{riordan} does not yield a closed-form expression for $c_n(1,t)$, but it will prove useful later for studying its arithmetic properties.

\section{Information about $\frac{f_1 f_6}{f_2^2 f_3}$}\label{pre2S}

Let $a(n)$ be defined by the generating function identity 
\begin{equation}
\label{genfn_an} 
A(q):=\sum_{n=0}^\infty a(n)q^n = \frac{f_1 f_6}{f_2^2 f_3}.
\end{equation}
This is the function which appears in \eqref{exform1} above. 
Our goal in this section is to prove a number of Ramanujan--like congruences satisfied by $a(n)$.  These will then be used in Section \ref{a1S} to prove several arithmetic properties satisfied by $\modd(1,t;N)$ for various values of $t$ and $N$.  To begin this analysis, we list several $q$--series identities that will be used in our proofs below.  
\begin{lemma}
\label{lemma:f1_over_f3} 
We have 
$$\frac{f_1}{f_3} = \frac{f_2f_{16}f_{24}^2}{f_6^2f_8f_{48}} -q\frac{f_2f_8^2f_{12}f_{48}}{f_4f_6^2f_{16}f_{24}}.$$
\end{lemma}
\begin{proof}
See Hirschhorn \cite[(30.10.1)]{Hir17}.  
\end{proof}

\begin{lemma}
\label{lemma:f3_over_f1} 
We have 
$$\frac{f_3}{f_1} = \frac{f_4f_{6}f_{16}f_{24}^2}{f_2^2f_8f_{12}f_{48}} +q\frac{f_6f_8^2f_{48}}{f_2^2f_{16}f_{24}}.$$
\end{lemma}
\begin{proof}
See Hirschhorn \cite[(30.10.3)]{Hir17}.  
\end{proof}

\begin{lemma}
\label{lemma:f1f3} 
We have 
$${f_1f_3} = \frac{f_2f_8^2f_{12}^4}{f_4^2f_6f_{24}^2} -q\frac{f_4^4f_6f_{24}^2}{f_2f_{8}^2f_{12}^2}.$$
\end{lemma}
\begin{proof}
See Hirschhorn \cite[(30.12.1)]{Hir17}.  
\end{proof}

\begin{lemma}
\label{lemma:1_over_f1f3} 
We have 
$$\frac{1}{f_1f_3} = \frac{f_8^2f_{12}^5}{f_2^2f_4f_6^4f_{24}^2} +q\frac{f_4^5f_{24}^2}{f_2^4f_6^2f_{8}^2f_{12}}.$$
\end{lemma}
\begin{proof}
See Hirschhorn \cite[(30.12.3)]{Hir17}.  
\end{proof}

\begin{lemma}
\label{lemma:1_over_f1squaredf3squared} 
We have 
$$\frac{1}{f_1^2f_3^2} = \frac{f_8^5f_{24}^5}{f_2^5f_6^5f_{16}^2f_{48}^2} +2q\frac{f_4^4f_{12}^4}{f_2^6f_6^6} +4q^4\frac{f_4^2f_{12}^2f_{16}^4f_{48}^2}{f_2^5f_6^5f_8f_{24}}.$$
\end{lemma}
\begin{proof}
See Baruah and Ojah \cite[Lemma 2.2]{BO15}.  
\end{proof}

\begin{lemma}
\label{lemma:1_over_f1squared} 
We have 
$$\frac{1}{f_1^2} = \frac{f_8^5}{f_{2}^5f_{16}^2}+2q\frac{f_4^2f_{16}^2}{f_2^5f_{8}}.$$
\end{lemma}
\begin{proof}
See Berndt \cite[p. 40]{Ber91}.  
\end{proof}

We will also require a few 3--dissection results in our work below.

\begin{lemma}
\label{lemma:f1squared_over_f2}
We have 
$$\frac{f_1^2}{f_2} = \varphi(-q) = \varphi(-q^3) -2q\frac{f_3f_{18}^2}{f_6f_9}$$
where 
$$\varphi(q) = 1 + 2\sum_{k=1}^\infty q^{k^2}$$ 
is one of Ramanujan's famous theta functions.
\begin{proof} 
See Hirschhorn \cite[(1.5.8)]{Hir17} and Andrews, Hirschhorn, and Sellers \cite[Lemma 2.6]{AHS10}.
\end{proof}
\end{lemma}

\begin{lemma}
\label{lemma:psi_3_dissection}
We have 
$$\psi(q) = P(q^3) + q\psi(q^9)$$
where 
$$\psi(q) = \sum_{k=0}^\infty q^{k(k+1)/2}$$ 
is one of Ramanujan's famous theta functions and 
$$P(q) = 1+q+q^2+q^5+q^7+q^{12}+q^{15}+ \dots$$
whose exponents are the generalized pentagonal numbers. 
\begin{proof} 
See Hirschhorn and Sellers \cite[p. 274]{HS10}.
\end{proof}
\end{lemma}

\begin{lemma}
\label{lemma:daSilva_Sellers_3diss}
We have 
\begin{align*}
\frac{f_2}{f_1^2} 
&= 
\frac{f_6^4f_9^6}{f_3^8f_{18}^3} + 2q\frac{f_6^3f_9^3}{f_3^7} + 4q^2\frac{f_6^2f_{18}^3}{f_3^6}, \textrm{\ \ and}   \\
\frac{f_2^3}{f_1^3} 
&= 
\frac{f_6}{f_3} + 3q\frac{f_6^4f_9^5}{f_3^8f_{18}} + 6q^2\frac{f_6^3f_9^2f_{18}^2}{f_3^7} + 12q^3\frac{f_6^2f_{18}^5}{f_3^6f_9}.
\end{align*}
\end{lemma}
\begin{proof}
See da Silva and Sellers \cite[Lemma 2]{dSS20}.
\end{proof}

\begin{lemma}
\label{lemma:AHS10_3diss}
We have 
$$
\frac{f_4}{f_1} = \frac{f_{12}f_{18}^4}{f_3^3f_{36}^2} + q\frac{f_6^2f_9^3f_{36}}{f_3^4f_{18}^2} + 2q^2\frac{f_6f_{18}f_{36}}{f_3^3}
$$
\end{lemma}
\begin{proof}
See Andrews, Hirschhorn, and Sellers \cite[Theorem 3.1]{AHS10}.  
\end{proof}

Next, define the functions $a(q)$, $b(q)$, and $c(q)$ by 
\begin{align*}
a(q) &:= \sum_{m,n = -\infty}^\infty q^{m^2+mn+n^2},  \\
b(q) &:=\sum_{m,n = -\infty}^\infty \omega^{m-n} q^{m^2+mn+n^2}, \textrm{\ \ \ and} \\
c(q) &:= q^{1/3}\sum_{m,n = -\infty}^\infty q^{m^2+mn+n^2+m+n}
\end{align*}
where $\omega$ is a cube root of unity other than 1.  These functions were introduced by Borwein, Borwein, and Garvan \cite{BBG94}.  

Several properties of the function $b(q)$ are known.  

\begin{lemma}
\label{bq_prod_representation}
We have 
$$b(q) = \frac{f_1^3}{f_3}.$$ 
\end{lemma}
\begin{proof}
See Hirschhorn \cite[(22.1.6)]{Hir17}.  
\end{proof}

\begin{lemma}
\label{bq_sum_representation}
We have 
$$b(q) = 1 - 3\sum_{n=-\infty}^\infty \frac{q^{3n+1}\left( 1-2q^{6n+2}\right)}{1-q^{9n+3}}.$$
\end{lemma}  
\begin{proof}
See Hirschhorn \cite[(27.1.9)]{Hir17}.  
\end{proof}
 
\begin{lemma}
\label{bq_2diss}
We have 
$$b(q) = b(q^4) -3q\psi(q^6)\left( \psi(q^2) - 3q^2\psi(q^{18}) \right).$$
\end{lemma}  
\begin{proof}
See Hirschhorn \cite[(22.6.10)]{Hir17}.  
\end{proof}
\begin{lemma}
\label{bq_aqcubed_and_cqcubed}
We have 
$$b(q) = a(q^3) - 3q\frac{f_9^3}{f_3}.$$
\end{lemma}  
\begin{proof}
See Hirschhorn \cite[(22.1.5) and (22.1.7)]{Hir17}.  
\end{proof}

We are now in a position to find the 2--dissection of the generating function for $a(n)$.  

\begin{theorem}
\label{genfn_a2n}
We have 
\begin{align}
\sum_{n=0}^\infty a(2n)q^n 
&= 
\frac{f_8f_{12}^2}{f_1f_3f_4f_{24}}, \label{a2n0-dissection} \textrm{\ \ and} \\
\sum_{n=0}^\infty a(2n+1)q^n 
&= 
-\frac{f_4^2f_6f_{24}}{f_1f_2f_3f_8f_{12}}.\label{a2n1-dissection} 
\end{align}
\end{theorem}
\begin{proof}
From \eqref{genfn_an} and Lemma \ref{lemma:f1_over_f3}, we know 
\begin{align*}
\sum_{n=0}^\infty a(2n)q^{2n} 
&= 
\frac{f_6}{f_2^2}\left( \frac{f_2f_{16}f_{24}^2}{f_6^2f_8f_{48}}  \right) \\
&= 
\frac{f_{16}f_{24}^2}{f_2f_6f_8f_{48}}.
\end{align*}
Replacing $q^2$ by $q$ throughout yields \eqref{a2n0-dissection}.  Again thanks to \eqref{genfn_an} and Lemma \ref{lemma:f1_over_f3}, we know 
\begin{align*}
\sum_{n=0}^\infty a(2n+1)q^{2n+1} 
&= 
\frac{f_6}{f_2^2}\left(  -q\frac{f_2f_8^2f_{12}f_{48}}{f_4f_6^2f_{16}f_{24}} \right)  \\
&= 
-q \frac{f_8^2f_{12}f_{48}}{f_2f_4f_6f_{16}f_{24}}.
\end{align*}
Dividing both sides by $q$, and then replacing $q^2$ by $q$ throughout yields \eqref{a2n1-dissection}.
\end{proof}

With the above in hand, we can now state a characterization modulo 2 satisfied by $a(2n)$ for all $n\geq 0$.

\begin{theorem}
\label{a2n_parity_characterization}
We have 
$$
a(2n)\equiv \begin{cases}
			1 \pmod{2}, & \text{if $n=0$ or $n$ is a square not divisible by $3$}\\
            0 \pmod{2}, & \text{otherwise}.
		 \end{cases}
$$
\end{theorem}
\begin{proof}
From Theorem \ref{genfn_a2n}, we know 
\begin{align*}
\sum_{n=0}^\infty a(2n)q^n 
&= 
\frac{f_8f_{12}^2}{f_1f_3f_4f_{24}} \\
&\equiv 
\frac{f_1^8f_{12}^2}{f_1f_3f_1^4f_{12}^2} \pmod{2}\\
&= 
\frac{f_1^3}{f_3}  \\
&= 
b(q) \ \ \ \textrm{thanks to Lemma \ref{bq_prod_representation}}   \\
&\equiv 
1 + \sum_{n=-\infty}^{\infty}  \frac{q^{3n+1}}{1-q^{9n+3}} \pmod{2} \ \ \ \textrm{thanks to Lemma  \ref{bq_sum_representation}}\\
&\equiv 
1 + \sum_{n\geq1,\; 3 \nmid n} q^{n^2} \pmod{2}  
\end{align*}
and this yields our result.
At the last step we applied
\begin{align*}
\sum_{n=-\infty}^{\infty} \frac{q^{3n+1}}{1-q^{9n+3}}&=\sum_{n=0}^{\infty} \frac{q^{3n+1}}{1-q^{9n+3}}+\sum_{n=0}^{\infty} \frac{q^{-3n-2}}{1-q^{-9n-6}}\\
&=\sum_{m,n\geq 0} q^{3n+1}q^{m(9n+3)}
-\sum_{n=0}^{\infty} \frac{q^{6n+4}}{1-q^{9n+6}}\\
&=\sum_{m,n\geq 0}  q^{(3n+1)(3m+1)}
-\sum_{m,n\geq 0}   q^{(3n+2)(3m+2)}\\
&=\sum_{n=0}^{\infty} q^{(3n+1)^2}-
\sum_{n=0}^{\infty}q^{(3n+2)^2}\\
&\qquad+2\sum_{0\leq n<m} q^{(3n+1)(3m+1)}
-2\sum_{0\leq n<m}  q^{(3n+2)(3m+2)}.
\end{align*}
\end{proof}

This now leads to numerous corollaries modulo 2, including the following congruences.
\begin{corollary}
\label{a6nmod2}
For all $n\geq 0$, 
\begin{align*}
a(6n+4) &\equiv 0 \pmod{2}, \\
a(6n+6) &\equiv 0 \pmod{2}.
\end{align*}
\end{corollary}
\begin{proof}
Note that there are no squares of the form $3n+2$ since 2 is a quadratic nonresidue modulo 3.  This yields the result for $2(3n+2) = 6n+4$.  Moreover, for $n\geq 0$, no number of the form $3n+3$ appears as an exponent on the right--hand side of Theorem \ref{a2n_parity_characterization}.  This completes the proof.  
\end{proof}
\begin{corollary}
\label{a8nmod2}
For all $n\geq 0$, 
\begin{align*}
a(8n+4) &\equiv 0 \pmod{2}, \\
a(8n+6) &\equiv 0 \pmod{2}.
\end{align*}
\end{corollary}
\begin{proof}
Note that there are no squares of the form $4n+2$ nor $4n+3$ because all squares are congruent to either 0 or 1 modulo 4.  The result follows from Theorem \ref{a2n_parity_characterization}.
\end{proof}
\begin{corollary}
\label{a2pnmod2}
For all $n\geq 0$, and for all primes $p\geq 5$,
\begin{align*}
a(2(pn+r)) &\equiv 0 \pmod{2}
\end{align*}
where $r$ is a quadratic nonresidue modulo $p$ with $1\leq r\leq p-1$.  
\end{corollary}
\begin{proof}
This follows immediately from Theorem \ref{a2n_parity_characterization}.
\end{proof}
 
We require one additional parity result for our work below, which is the following.  

\begin{theorem}
\label{a24n13mod2}
For all $n\geq 0$, $a(24n+13) \equiv 0 \pmod{2}.$  
\end{theorem}
\begin{proof}
From Theorem \ref{genfn_a2n}, we know 
\begin{align*}
\sum_{n=0}^\infty a(2n+1)q^n 
&= 
-\frac{f_4^2f_6f_{24}}{f_1f_2f_3f_8f_{12}}\\
&\equiv 
\frac{f_1^8f_{3}^2f_{12}^2}{f_1f_2f_3f_1^8f_{12}} \pmod{2}\\
&= 
\frac{f_{12}}{f_2}\cdot \frac{f_3}{f_1}.
\end{align*}
Using Lemma \ref{lemma:f3_over_f1}, we then see that 
\begin{align*}
\sum_{n=0}^\infty a(4n+1)q^n 
&\equiv 
\frac{f_6}{f_1}\cdot \frac{f_2f_3f_8f_{12}^2}{f_1^2f_4f_6f_{24}} \pmod{2} \\
&\equiv 
\frac{f_3}{f_1}\cdot \frac{f_6f_2f_4^2f_{12}^2}{f_2f_4f_6f_{12}^2} \pmod{2} \\
&= 
\frac{f_3}{f_1}\cdot f_4.
\end{align*}
Thanks again to Lemma \ref{lemma:f3_over_f1}, we then see that 
\begin{align*}
\sum_{n=0}^\infty a(8n+5)q^n 
&\equiv 
f_2\cdot \frac{f_3f_4^2f_{24}}{f_1^2f_8f_{12}}   \pmod{2}   \\
&\equiv 
\frac{f_2f_3f_2^4f_{12}^2}{f_2f_2^4f_{12}} \pmod{2} \\
&= 
{f_3f_{12}}
\end{align*}
which is a function of $q^3$.  Thus, for all $n\geq 0$,  $a(8(3n+1)+5) = a(24n+13)\equiv 0 \pmod{2}$.
\end{proof}

Next, we wish to prove the following congruences modulo 4 that are satisfied by $a(n)$.  

\begin{theorem}  
For all $n\geq 0$, 
\begin{align}
a(12n+6) &\equiv 0 \pmod{4}, \label{12n6_mod4} \\
a(16n+6) &\equiv 0 \pmod{4}, \label{16n6_mod4} \\
a(24n+16) &\equiv 0 \pmod{4}, \label{24n6_mod4} \quad\textrm{and}  \\
a(24n+22) &\equiv 0 \pmod{4}. \label{24n22_mod4}   
\end{align}
\end{theorem}

\begin{proof}
We begin the proof of this theorem by returning to Theorem \ref{genfn_a2n} and performing another 2--dissection.  Namely, from Lemma \ref{lemma:1_over_f1f3}, we have 
\begin{align*}
\sum_{n=0}^\infty a(2n)q^n 
&= 
\frac{f_8f_{12}^2}{f_4f_{24}}\left( \frac{1}{f_1f_3} \right) \\
&= 
\frac{f_8f_{12}^2}{f_4f_{24}}\left( \frac{f_8^2f_{12}^5}{f_2^2f_4f_6^4f_{24}^2} +q\frac{f_4^5f_{24}^2}{f_2^4f_6^2f_{8}^2f_{12}} \right).
\end{align*}
We then immediately see that 
\begin{equation}
\label{genfn_a4n}
\sum_{n=0}^\infty a(4n)q^n = \frac{f_4^3f_6^7}{f_1^2f_2^2f_3^4f_{12}^3}
\end{equation} 
and 
\begin{equation}
\label{genfn_a4n2}
\sum_{n=0}^\infty a(4n+2)q^n = \frac{f_2^4f_6f_{12}}{f_1^4f_3^2f_{4}}
\end{equation} 
after elementary simplifications.  

We begin by rewriting \eqref{genfn_a4n} as  
\begin{align*}
\sum_{n=0}^\infty a(4n)q^n 
&= 
\frac{f_6^7}{f_3^4f_{12}^3} \cdot \frac{f_2}{f_1^2} \cdot \frac{f_4^3}{f_2^3}
\end{align*}
which allows us to perform a 3--dissection thanks to Lemma \ref{lemma:daSilva_Sellers_3diss}.  In this vein, we have 
\begin{align*}
\sum_{n=0}^\infty a(4(3n+1))q^{3n+1} 
& \equiv  
\frac{f_6^7}{f_3^4f_{12}^3} \left( \frac{f_6^4f_9^6}{f_3^8f_{18}^3}\cdot 6q^4\frac{f_{12}^3f_{18}^2f_{36}^2}{f_6^7} + 2q\frac{f_6^3f_9^3}{f_3^7}\cdot\frac{f_{12}}{f_6}   \right)   \pmod{4} \\
&= 
6q^4\frac{f_6^4f_9^6f_{36}^2}{f_3^{12}f_{18}} + 2q\frac{f_6^9f_9^3}{f_3^{11}f_{12}^2}
\end{align*}
which implies 
\begin{align*}
\sum_{n=0}^\infty a(12n+4)q^{n} 
& \equiv  
2q\frac{f_2^4f_3^6f_{12}^2}{f_1^{12}f_{6}} + 2\frac{f_2^9f_3^3}{f_1^{11}f_{4}^2} \pmod{4} \\
& \equiv 
2q\frac{f_2^4f_6^3f_{12}^2}{f_2^{6}f_{6}} + 2\frac{f_2^9f_6}{f_2^{5}f_{4}^2}\cdot\frac{f_3}{f_1} \pmod{4} \\
&= 
2q\frac{f_2^4f_6^3f_{12}^2}{f_2^{6}f_{6}} + 2\frac{f_2^9f_6}{f_2^{5}f_{4}^2}\left(  \frac{f_4f_{6}f_{16}f_{24}^2}{f_2^2f_8f_{12}f_{48}} +q\frac{f_6f_8^2f_{48}}{f_2^2f_{16}f_{24}}  \right) 
\end{align*}
using Lemma \ref{lemma:f3_over_f1}.
This means that 
\begin{align*}
\sum_{n=0}^\infty a(12(2n+1)+4)q^{2n+1} 
& \equiv  
2q\frac{f_2^4f_6^3f_{12}^2}{f_2^{6}f_{6}} + 2q\frac{f_2^9f_6}{f_2^{5}f_{4}^2}\cdot  \frac{f_6f_8^2f_{48}}{f_2^2f_{16}f_{24}}  \pmod{4}  \\
&= 
2q\frac{f_6^2f_{12}^2}{f_2^2} +2q\frac{f_2^2f_6^2f_8^2f_{48}}{f_4^2f_{16}f_{24}} \\
& \equiv 
2q\frac{f_3^4f_3^8}{f_1^4} + 2q\frac{f_1^4f_3^4f_1^{16}f_3^{16}}{f_1^8f_1^{16}f_3^8} \pmod{4} \\
& \equiv 
2q\frac{f_3^{12}}{f_1^4} + 2q\frac{f_3^{12}}{f_1^4}  \pmod{4}  \\
& \equiv 
0 \pmod{4}.
\end{align*}
This proves \eqref{24n6_mod4}.  

Next, from \eqref{genfn_a4n2} and the fact that $f_1^4 \equiv f_2^2 \pmod{4}$, we have 
\begin{align}
\sum_{n=0}^\infty a(4n+2)q^n
&\equiv 
 \frac{f_2^2f_6f_{12}}{f_3^2f_{4}} \pmod{4} \label{4n2mod4-intermediate} \\
&= 
 \frac{f_2^2f_6f_{12}}{f_{4}}\left( \frac{f_{24}^5}{f_{6}^5f_{48}^2}+2q^3\frac{f_{12}^2f_{48}^2}{f_6^5f_{24}} \right)  \notag   
\end{align}
 where we also applied Lemma \ref{lemma:1_over_f1squared}.
From the above, we see that
$$
\sum_{n=0}^\infty a(4(2n+1)+2) q^{2n+1} \equiv   2q^3\frac{f_2^2f_6f_{12}}{f_4}\cdot \frac{f_{12}^2f_{48}^2}{f_6^5f_{24}}  \pmod{4}
$$
which means 
\begin{align*}
\sum_{n=0}^\infty a(8n+6) q^{n} 
&\equiv   
2q\frac{f_1^2f_6^3f_{24}^2}{f_2f_3^4f_{12}} \pmod{4} \\
&\equiv   
2q\frac{f_2f_6^3f_{12}^4}{f_2f_6^2f_{12}} \pmod{4} \\
&\equiv 
2qf_6f_{12}^3 \pmod{4}. 
\end{align*}
Note that this last expression, when written as a power series in $q$, contains only odd powers of $q$.  Therefore, for all $n\geq 0$, 
$a(8(2n)+6) = a(16n+6) \equiv 0 \pmod{4}$, 
which is \eqref{16n6_mod4}.  In addition, the expression $2qf_6f_{12}^3$ can also be viewed as a function of $q^{3n+1}$, and this means that, for all $n\geq 0$, $a(8(3n+2)+6) = a(24n+22) \equiv 0 \pmod{4}$.  This is \eqref{24n22_mod4}.    

Returning to \eqref{4n2mod4-intermediate}, by Lemma \ref{lemma:f1squared_over_f2} we see that 
\begin{align*}
\sum_{n=0}^\infty a(4n+2)q^n 
&\equiv 
\frac{f_2^2f_6f_{12}}{f_3^2f_{4}} \pmod{4} \\
&\equiv 
\frac{f_6f_{12}}{f_3^2} \left(   \varphi(-q^6) -2q^2\frac{f_6f_{36}^2}{f_{12}f_{18}}  \right)       \pmod{4}
\end{align*}
and this last expression, when written as a power series in $q$, contains no powers of the form $q^{3n+1}$.  Therefore, for all $n\geq 0$, 
$$a(4(3n+1)+2) = a(12n+6) \equiv 0 \pmod{4}.$$  
This is \eqref{12n6_mod4}, and this completes the proof of the theorem.  
\end{proof}
 
We now transition to proving a pair of congruences modulo 8 satisfied by $a(n)$.  
\begin{theorem}
\label{12n9_24n19_mod8}
For all $n\geq 0$, 
\begin{align*}
a(12n+9) & \equiv 0 \pmod{8}, \textrm{\ \ and} \\
a(24n+19) & \equiv 0 \pmod{8}.
\end{align*}
\end{theorem}
 
\begin{proof}

From \eqref{a2n1-dissection}, we know 
\begin{align}
\sum_{n=0}^\infty a(2n+1)q^n 
&= 
-\frac{f_4^2f_6f_{24}}{f_2f_8f_{12}}\cdot \frac{1}{f_1f_3} \notag \\
&= 
-\frac{f_4^2f_6f_{24}}{f_2f_8f_{12}}\left( \frac{f_8^2f_{12}^5}{f_2^2f_4f_6^4f_{24}^2} +q\frac{f_4^5f_{24}^2}{f_2^4f_6^2f_{8}^2f_{12}} \right) \label{a2n1-diss-intermediate}
\end{align}
using Lemma \ref{lemma:1_over_f1f3}.
Thus, we have 
\begin{align*}
\sum_{n=0}^\infty a(4n+1)q^n 
&=
-\frac{f_2^2f_3f_{12}f_4^2f_6^5}{f_1f_4f_6f_1^2f_2f_3^4f_{12}^2} \\
&= 
-\frac{f_6^4}{f_3^3f_{12}}\cdot \frac{f_2f_4}{f_1^3} \\
&= 
-\frac{f_6^4}{f_3^3f_{12}}\cdot \frac{f_2}{f_1^2}\cdot\frac{f_4}{f_1}.  
\end{align*}
Using Lemma \ref{lemma:daSilva_Sellers_3diss} and Lemma \ref{lemma:AHS10_3diss}, we know 
\begin{align*}
&\sum_{n=0}^\infty a(4(3n+2)+1)q^{3n+2}  \\
&= 
-\frac{f_6^4}{f_3^3f_{12}}\left(  2q^2\frac{f_6^4f_9^6}{f_3^8f_{18}^3}\frac{f_6f_{18}f_{36}}{f_3^3} + 2q^2\frac{f_6^3f_9^3}{f_3^7}\frac{f_6^2f_9^3f_{36}}{f_3^4f_{18}^2} + 4q^2\frac{f_6^2f_{18}^3}{f_3^6}\frac{f_{12}f_{18}^4}{f_3^3f_{36}^2}\right).
\end{align*}
This yields 
\begin{align*}
\sum_{n=0}^\infty a(12n+9)q^{n}  
&= 
-\frac{f_2^4}{f_1^3f_{4}}\left(  2\frac{f_2^5f_3^6f_{12}}{f_1^{11}f_6^2}+2\frac{f_2^5f_3^6f_{12}}{f_1^{11}f_6^2}+ 4\frac{f_2^2f_4f_6^7}{f_1^9f_{12}^2}\right) \\
&= 
-\frac{f_2^4}{f_1^3f_{4}}\left(  4\frac{f_2^5f_3^6f_{12}}{f_1^{11}f_6^2}+ 4\frac{f_2^2f_4f_6^7}{f_1^9f_{12}^2}\right) \\
&\equiv 
-\frac{f_2^4}{f_1^3f_{4}}\left(  4\frac{f_3^6}{f_1}+ 4\frac{f_3^6}{f_1}\right) \pmod{8} \\
&\equiv 
0 \pmod{8}.
\end{align*}
This proves the first congruence in this theorem.  
Returning to \eqref{a2n1-diss-intermediate}, we can also see that 
\begin{align*}
\sum_{n=0}^\infty a(4n+3)q^{n}  
&= 
-\frac{f_2^2f_3f_{12}}{f_1f_4f_{6}}\cdot \frac{f_2^5f_{12}^2}{f_1^4f_3^2f_{4}^2f_{6}} \\
&= 
-\frac{f_2^7f_{12}^3}{f_1^5f_3f_{4}^3f_{6}^2} \\
&= 
-\frac{f_2^7f_{12}^3}{f_1^8f_{4}^3f_{6}^2} b(q) \\
& \equiv 
-\frac{f_2^7f_{12}^3}{f_2^4f_{4}^3f_{6}^2} b(q) \pmod{8} \\
&= 
-\frac{f_2^3f_{12}^3}{f_{4}^3f_{6}^2}\left( b(q^4) -3q\psi(q^6)\left( \psi(q^2) - 3q^2\psi(q^{18}) \right) \right)
\end{align*}
from Lemma \ref{bq_2diss}.  Thus, we have 
\begin{align*}
\sum_{n=0}^\infty a(4(2n)+3)q^{2n}  
&\equiv 
-\frac{f_2^3f_{12}^3}{f_{4}^3f_{6}^2} b(q^4) \pmod{8}
\end{align*}
which means 
\begin{align*}
\sum_{n=0}^\infty a(8n+3)q^{n}  
&\equiv 
-\frac{f_1^3f_{6}^3}{f_{2}^3f_{3}^2} b(q^2) \pmod{8} \\
&= 
-\frac{f_1^3f_{6}^3}{f_{2}^3f_{3}^2}\cdot \frac{f_2^3}{f_6} \\
&= 
-\frac{f_1^3f_{6}^2}{f_{3}^2}  \\
&= 
-\frac{f_{6}^2}{f_{3}}b(q)  \\
&= 
-\frac{f_{6}^2}{f_{3}}\left( a(q^3) - 3q\frac{f_9^3}{f_3} \right) \textrm{\ \ from Lemma \ref{bq_aqcubed_and_cqcubed}}.
\end{align*}
Given the definition of $a(q)$ above, we know that $a(q^3)$ is a function of $q^3$.  Therefore, when written as a power series in $q$, the expression 
$$
-\frac{f_{6}^2}{f_{3}}\left( a(q^3) - 3q\frac{f_9^3}{f_3} \right)
$$
contains no terms of the form $q^{3n+2}$.  From the string of congruences and equalities above, we then know that, for all $n\geq 0$, $a(8(3n+2)+3) = a(24n+19) \equiv 0 \pmod{8}.$  This completes the proof of the theorem.  
\end{proof}

We can prove one additional congruence modulo 8 satisfied by the function $a(n)$.  

\begin{theorem}
\label{32n28mod8}
For all $n\geq 0$, $a(32n+28) \equiv 0 \pmod{8}$. 
\end{theorem}
\begin{proof}
From \eqref{genfn_a4n}, we know 
$$
\sum_{n=0}^\infty a(4n)q^n 
= 
\frac{f_4^3f_6^7}{f_2^2f_{12}^3}\cdot \frac{1}{f_1^2f_3^4}.
$$
Using Lemma \ref{lemma:1_over_f1squaredf3squared} and Lemma \ref{lemma:1_over_f1squared}, we then have 
\begin{align*}
 \sum_{n=0}^\infty a(8n+4)q^{2n+1}&= \frac{f_4^3f_6^7}{f_2^2f_{12}^3} 
\cdot\left( 2q^3\frac{f_8^5f_{24}^5}{f_2^5f_6^5f_{16}^2f_{48}^2}\cdot \frac{f_{12}^2f_{48}^2}{f_6^5f_{24}} + 2q\frac{f_4^4f_{12}^4}{f_2^6f_6^6}\cdot \frac{f_{24}^5}{f_6^5f_{48}^2}\right.\\
&\qquad \qquad\qquad
 \left. + 8q^7\frac{f_4^2f_{12}^2f_{16}^4f_{48}^2}{f_2^5f_6^5f_8f_{24}}\cdot \frac{f_{12}^2f_{48}^2}{f_6^5f_{24}}  \right)  \\
&\equiv 
2q^3\frac{f_4^3f_8^5f_{24}^4}{f_2^7f_6^3f_{12}f_{16}^2} + 2q\frac{f_4^7f_{12}f_{24}^5}{f_2^8f_6^4f_{48}^2} \pmod{8}
\end{align*}
after simplification.  Dividing both sides of the above by $q$ and then replacing $q^2$ by $q$ throughout yields 
\begin{align}
\sum_{n=0}^\infty a(8n+4)q^{n} 
&\equiv 
2q\frac{f_2^3f_4^5f_{12}^4}{f_1^7f_3^3f_{6}f_{8}^2} + 2\frac{f_2^7f_{6}f_{12}^5}{f_1^8f_3^4f_{24}^2} \pmod{8} \notag \\
&\equiv 
2q\frac{f_2f_4f_{12}^4}{f_1^3f_3^3f_6} + 2\frac{f_1^4f_2f_{12}}{f_6} \pmod{8} \notag \\
&\equiv 
2q\frac{f_2f_4f_{12}^4}{f_1^3f_3^3f_6} + 2\frac{f_2^3f_{12}}{f_6} \pmod{8} \label{temp_mod4}\\
&\equiv 
2q\frac{f_4f_{12}^4}{f_2f_6^3}\left(f_1f_3\right) + 2\frac{f_2^3f_{12}}{f_6} \pmod{8} \notag
\end{align}
utilizing the fact that $f_2^2 \equiv f_1^4 \pmod{4}$ several times.  
We now apply Lemma \ref{lemma:f1f3} to obtain 
$$
\sum_{n=0}^\infty a(8n+4)q^{n} 
\equiv 
2q\frac{f_4f_{12}^4}{f_2f_6^3}\left( \frac{f_2f_8^2f_{12}^4}{f_4^2f_6f_{24}^2} -q\frac{f_4^4f_6f_{24}^2}{f_2f_{8}^2f_{12}^2}  \right) + 2\frac{f_2^3f_{12}}{f_6} 
\pmod{8}.
$$
From the above congruence, we see that 
\begin{align*}
\sum_{n=0}^\infty a(8(2n+1)+4)q^{2n+1} 
&\equiv 
2q\frac{f_4f_{12}^4}{f_2f_6^3}\cdot \frac{f_2f_8^2f_{12}^4}{f_4^2f_6f_{24}^2} \pmod{8} \\
&\equiv 
2q\frac{f_8^2f_{12}^6}{f_4f_{24}^2} \pmod{8}
\end{align*}
after simplification.  Dividing both sides of the above by $q$ and replacing $q^2$ by $q$ yields 
$$
\sum_{n=0}^\infty a(16n+12)q^{n} \equiv 2\frac{f_4^2f_{6}^6}{f_2f_{12}^2} \pmod{8}
$$
and this last expression is an even function in $q$.  Thus, for all $n\geq 0$, 
$$a(16(2n+1)+12) = a(32n+28) \equiv 0 \pmod{8}.$$
\end{proof}

We share one last congruence satisfied by $a(n)$, again modulo 4.  

\begin{theorem}
\label{32n20mod4}
For all $n\geq 0$, $a(32n+20) \equiv 0 \pmod{4}$. 
\end{theorem}
\begin{proof}
Returning to \eqref{temp_mod4}, and recognizing that divisibility by 8 immediately implies divisibility by 4, we have 
\begin{align*}
\sum_{n=0}^\infty a(8n+4)q^{n} 
&\equiv 
2q\frac{f_2f_4f_{12}^4}{f_1^3f_3^3f_6} + 2\frac{f_2^3f_{12}}{f_6} \pmod{4} \\
&\equiv 
2q\frac{f_1^2f_4f_3^4f_{12}^3}{f_1^3f_3^5}+ 2\frac{f_2^3f_{6}^2}{f_6} \pmod{4} \\
&\equiv 
2qf_4f_{12}^3\cdot \frac{1}{f_1f_3}+2f_2^3f_{6} \pmod{4} \\
&= 
2qf_4f_{12}^3\left( \frac{f_8^2f_{12}^5}{f_2^2f_4f_6^4f_{24}^2} +q\frac{f_4^5f_{24}^2}{f_2^4f_6^2f_{8}^2f_{12}} \right) +2f_2^3f_{6}
\end{align*}
using Lemma \ref{lemma:1_over_f1f3}.  Simplifying further modulo 4, we have 
$$
\sum_{n=0}^\infty a(8n+4)q^{n}  
\equiv 
2q\frac{f_4^3f_{12}^6}{f_{24}^2} + 2q^2\frac{f_4^4f_{12}f_{24}^2}{f_8^2} + 2f_2^3f_6 \pmod{4}
$$
which implies that 
$$
\sum_{n=0}^\infty a(16n+4)q^{2n}  
\equiv 
2q^2\frac{f_4^4f_{12}f_{24}^2}{f_8^2} + 2f_2^3f_6 \pmod{4}.
$$
Replacing $q^2$ by $q$ and then continuing to simplify modulo 4, we have 
\begin{align*}
\sum_{n=0}^\infty a(16n+4)q^{n}  
&\equiv 
2q\frac{f_2^4f_{6}f_{12}^2}{f_4^2} + 2f_1^3f_3 \pmod{4} \\
&\equiv 
2q\frac{f_2^4f_{6}f_{12}^2}{f_4^2} + 2f_2f_6\cdot \frac{f_1}{f_3} \pmod{4} \\
&= 
2q\frac{f_2^4f_{6}f_{12}^2}{f_4^2} + 2f_2f_6\left( \frac{f_2f_{16}f_{24}^2}{f_6^2f_8f_{48}} -q\frac{f_2f_8^2f_{12}f_{48}}{f_4f_6^2f_{16}f_{24}} \right) 
\end{align*}
from Lemma \ref{lemma:f1_over_f3}.  We then see that 
\begin{align*}
\sum_{n=0}^\infty a(16(2n+1)+4)q^{2n+1}  
&\equiv 
2q\frac{f_2^4f_{6}f_{12}^2}{f_4^2} -2q\frac{f_2^2f_8^2f_{12}f_{48}}{f_4f_6f_{16}f_{24}}  \pmod{4} \\
&\equiv 
2q{f_{6}f_{12}^2} -2q\frac{f_{12}f_{48}}{f_6f_{24}}  \pmod{4} \\
&\equiv 
2q{f_{6}f_{12}^2} -2q\frac{f_{6}^2f_{12}^4}{f_6f_{12}^2}  \pmod{4} \\
&\equiv 
2q{f_{6}f_{12}^2} -2q{f_{6}f_{12}^2}  \pmod{4} \\
&\equiv 
0 \pmod{4}.
\end{align*}
Thus, for all $n\geq 0$, 
$$
a(16(2n+1)+4) = a(32n+20) \equiv 0 \pmod{4}.
$$
\end{proof}

\section{Congruences for coefficients $c_n(a,t)$}\label{coefficientsS}

As our focus turns to the divisibility of binomial coefficients by powers of primes (mainly $2$ and $3$), we recall a fundamental result, namely Kummer's Theorem (see, for instance, \cite{Granville95}): if $p$ is a prime, then the largest power of $p$ dividing $\binom{n}{m}$ is  equal to
$$\nu_p \left(\binom{n}{m}\right)
=\frac{S_p(m)+S_p(n-m)-S_p(n)}{p-1}.$$
where $\nu_p(k)$ is the largest power of $p$ dividing an integer $k$
and $S_p(k)$ is the sum of its base-$p$ digits. 

We also require the following Lemma, which plays a key role in studying congruences modulo powers of two.

\begin{lemma} For any integer $s\geq 1$,
\begin{equation}\label{power2cong}
(1+z)^{2^s}\equiv (1+z^2)^{2^{s-1}}\pmod{2^s}.
\end{equation}
\end{lemma}
\begin{proof} We show the claim by induction with respect to the exponent $s$. For $s=1$, it holds:
$$(1+z)^{2}=1+2z+z^2\equiv 1+z^2\pmod{2}.$$
Let us now proceed to the inductive step. If the statement holds for $s$ then
$$(1+z)^{2^s}=(1+z^2)^{2^{s-1}}+2^{s} P(z)$$
for some polynomial $P$ with integer coefficients. This implies that
\begin{align*}
(1+z)^{2^{s+1}}&=\left((1+z)^{2^{s}}\right)^2\\
&=\left((1+z^2)^{2^{s-1}}+2^{s}P(z)\right)^2\\
&=(1+z^2)^{2^{s}}+2(1+z^2)^{2^{s-1}}\cdot 2^{s} P(z)+2^{2s} P(z)^2 \\
&\equiv(1+z^2)^{2^{s}}\pmod{2^{s+1}}
\end{align*}
and we are done.
\end{proof}

We now present three theorems, each establishing congruences for the coefficients $c_n(a,t)$ for $a=-2,0,1$, in that order. 

\begin{theorem}\label{Tcong-cm2} Let
$$c_n(-2,t)=(-1)^{n+t}\frac{2n}{n+t}\binom{n+t}{2t}.$$
Then \smallskip

\noindent 1)  if $s\geq 1$, $J\geq 1$, and $n\geq 0$ is even, then
$$c_n(-2,2^sJ-1)\equiv 0\pmod{2^{s+1}},$$

\noindent 2)  if $J\geq 0$ and $n\not\equiv 13,14 \pmod{27}$,  then
$$c_n(-2,27J+13)\equiv 0\pmod{3},$$

\noindent 3)  if $J\geq 1$ and $n\not\equiv \pm1 \pmod{27}$,  then
$$c_n(-2,27J-1)\equiv 0\pmod{3}.$$
\end{theorem}
\begin{proof} 1) Let $t=2^sJ-1$, then by \eqref{riordan} and  \eqref{power2cong} we have that
\begin{align*}
c_n(-2,2^sJ-1)&=[z^n]\frac{z^t-z^{t+2}}{(1+2z+z^2)^{t+1}}\\
&=[z^n]\frac{z^{2^sJ-1}-z^{2^sJ+1}}{(1+z)^{2^{s+1}J}}\\
&\equiv [z^n]\frac{z^{2^sJ-1}-z^{2^sJ+1}}{(1+z^2)^{2^{s}J}} \pmod{2^{s+1}}
\end{align*}
Since $\frac{z^{2^sJ-1}-z^{2^sJ+1}}{(1+z^2)^{2^{s}J}}$ is an odd function, the right-hand side is zero when $n$ is even.

\noindent 2) Let $t=27J+13$ and $n=27N+r$ with $0\leq r<27$.  Then
$$\nu_3(c_n(-2,t))=\nu_3\left(\frac{n}{t}\binom{n+t-1}{2t-1}\right)
=\nu_3(n)+\nu_3\left(\binom{n+t-1}{2t-1}\right).$$
If $n$ is a multiple of $3$, then the congruence trivially holds. 
Assume that $r\in \{1,2,4,5,7,8,10,11\}$. Then
\begin{align*}
\nu_3(c_n(-2,t))&=\nu_3\left(\binom{n+t-1}{2t-1}\right)\\
&=\frac{1}{2}\Big(S_3(2t-1)+S_3(n-t)-S_3(n+t-1)\Big)\\
&=\frac{1}{2}\Big(S_3(27\cdot 2J+25)+S_3(27(N-J-1)+r+14)\\
&\qquad\qquad-S_3(27(N+J)+r+12)\Big)\\
&=\frac{1}{2}\Big(S_3(2J)+5+S_3(N-J-1)+S_3(r+14)\\
&\qquad\qquad-S_3(N+J)-S_3(r+12)\Big)\\
&=3-\nu_3\left(\binom{r+14}{2}\right)+\nu_3(N-J)+\nu_3\left(\binom{N+J}{2J}\right)\\
&\geq 3-\nu_3(r+14)-\nu_3(r+13)\geq 1
\end{align*}
where, in the last step, we made use of the inequality 
$$\nu_3(r+14)+\nu_3(r+13)\leq \max(\nu_3(r+14),\nu_3(r+13))<3$$
which holds because $\gcd(r+14,r+13)=1$ and $r+14<3^3$.

\noindent Similarly, for $r\in \{16,17,19,20,22,23,25,26\}$,
\begin{align*}
\nu_3(c_n(-2,t))&=3-\nu_3\left(\binom{r-13}{2}\right)+\nu_3(N+1-J)+\nu_3\left(\binom{N+1+J}{2J}\right)\\
&\geq 3-\nu_3(r-14)-\nu_3(r-13)\geq 1.
\end{align*}

\noindent The proof of 3) is similar.
\end{proof}

\begin{theorem}\label{Tcong-c0} Let
$$c_n(0,t)=(-1)^{n+t-1} \frac{2n-1}{n+t}\binom{n+t}{2t+1}.$$
Then \smallskip

\noindent 1)  if $J\geq 1$ and $n\not\equiv 0,1 \pmod{4}$, then
\begin{align*}
c_n(0,4J-1)&\equiv 0\pmod{4}\\
c_n(0,8J-1)&\equiv 0\pmod{8},
\end{align*}

\noindent 2)  if $J\geq 1$ and $n\not\equiv 0,1 \pmod{8}$, then
\begin{align*}
c_n(0,32J-1)&\equiv 0\pmod{16}\\
c_n(0,64J-1)&\equiv 0\pmod{32},
\end{align*}

\noindent 3)  if $J\geq 0$ and $n\not\equiv 13,15 \pmod{27}$,  then
$$c_n(0,27J+12)\equiv 0\pmod{3},$$

\noindent 4)  if $J\geq 1$ and $n\not\equiv \pm1 \pmod{27}$,  then
$$c_n(0,27J-1)\equiv 0\pmod{3}.$$
\end{theorem}

\begin{proof} We note that
$$\frac{2n-1}{n+t}\binom{n+t}{2t+1}=\frac{2n-1}{2t+1}\binom{n+t-1}{2t}.$$
For 1) and 2), let $t=2^sJ-1$ and $n=2^sN+r$ with $1<r<2^s$.  Then
\begin{align*}
\nu_2(c_n(0,t))&=\nu_2\left(\binom{n+t-1}{2t}\right)\\
&=S_2(2t)+S_2(n-t-1)-S_2(n+t-1)\\
&=S_2(2^s(J-1)+2^s-1)+S_2(2^s(N-J)+r)\\
&\qquad\qquad -S_2(2^s(N+J)+r-2)\\
&=S_2(J-1)+s+S_2(N-J)+S_2(r)-S_2(N+J)-S_2(r-2)\\
&=s-\nu_2\left(\binom{r}{2}\right)+\nu_2\left(J\binom{N+J}{2J}\right)\\
&\geq s+1-\nu_2(r)-\nu_2(r-1)\geq s.
\end{align*}
Hence, if $n\not\equiv  0,1 \pmod{4}$, then
$$\nu_2(c_n(0,t))\geq s$$
which implies 1).
On the other hand, if $n\not\equiv  0,1$ modulo $8$ then
$$\nu_2(c_n(0,t))\geq s-1$$
which implies 2).

The proofs of 3) and 4) can be carried out in a similar way.
\end{proof}

\begin{theorem}\label{Tcong-c1} Let
$$c_n(1,t)=\sum_{k=t}^n(-1)^{n-k}\frac{2n}{n+k}\binom{n+k}{2k}\binom{k}{t}3^{k-t}.$$
Then \smallskip

\noindent 1)  if $J\geq 1$ and $n$ is even, then
$$c_n(1,2J-1)\equiv 0\pmod{2},$$
\noindent 2)  if $J\geq 1$ and $n\not\equiv \pm 1\pmod{2^{s-1}}$ with $s\geq 2$, then
$$c_n(1,2^sJ-1)\equiv 0\pmod{4},$$
\noindent 3)  if $J\geq 1$ and $n\not\equiv \pm 1\pmod{32}$, then
$$c_n(1,64J-1)\equiv 0\pmod{8}.$$
\end{theorem}
\begin{proof}\noindent If $k\geq t\geq 1$, then 
$$\frac{2n}{n+k}\binom{n+k}{2k}\binom{k}{t}=\frac{n}{t}\binom{n+k-1}{2k-1}\binom{k-1}{t-1},$$
which implies that if $t=2J-1$ and $n$ is even then 1) trivially holds.

\noindent 2) Let $t=2^sJ-1$. By \eqref{power2cong} for $s=2$,
$$(1+z^m)^4\equiv (1+z^{2m})^2\pmod{4},$$
and from \eqref{riordan} we obtain
\begin{align*}
c_n(1,t)&=[z^n]\frac{z^t-z^{t+2}}{(1-z+z^2)^{t+1}}\\
&=[z^n]\frac{(z^{2^sJ-1}-z^{2^sJ-1})(1+z)^{2^sJ}}{(1+z^3)^{2^sJ}}\\
&\equiv [z^n]\frac{(z^{2^sJ-1}-z^{2^sJ+1})(1+z^{2^{s-1}})^{2J}}{(1+z^{3\cdot 2^{s-1}})^{2J}}\pmod{4}
\end{align*}
where the right-hand side in $0$ as soon as $n\not\equiv \pm 1\pmod{2^{s-1}}$.

\noindent The argument for 3) proceeds in a similar manner after noting that, by \eqref{power2cong} for $s=3$,
$$(1+z^m)^8\equiv (1+z^{2m})^4\pmod{8}.$$
\end{proof}

We now have all of the tools needed to prove congruences satisfied by the functions $\modd(a,t;N)$ for various values of $a$, $t$, and $N$.  In the next three sections, we provide such results, categorized by the value of $a$.  

\section{Congruences for $a=-2$}\label{am2S}

We begin by noting that 
$$\modd(-2,1;N)\equiv
\begin{cases}
1 \pmod{2} &\text{$N$ is an odd square,}\\
0 \pmod{2}&\text{otherwise,}
\end{cases}
$$
and 
\begin{align*}
\modd(-2,1;6N+5) 
&\equiv 
\modd(1,1;6N+5) \pmod{3} \\
&=0
\end{align*}
using \eqref{modd1_1_6n5_equal_0} and \eqref{easy3}.  
Therefore, for all $N\geq 0$, 
$$\modd(-2,1;6N+5) =0 \pmod{6}.$$
Next, we recall from \cite{HirschhornSellers05} that,
if $8n+r>0$, then
\begin{equation}\label{ovc8}
\nu_2(\op(8n+r))\geq
\begin{cases}
1 &\text{if $r=0,1,4$,}\\
2 &\text{if $r=2$,}\\
3 &\text{if $r=3,5,6$,}\\
6 &\text{if $r=7$.}
\end{cases}
\end{equation}
and, if $9n+r>0$, then
$$
\nu_2(\op(9n+r))\geq
\begin{cases}
1 &\text{if $r=0,1,4,7$,}\\
2 &\text{if $r=2,5,8$,}\\
3 &\text{if $r=3,6$.}
\end{cases}
$$
Moreover, from \cite[(1.8)]{CHSZ24} and \cite[Theorem 2.1]{HirschhornSellers05b} we have, respectively,
\begin{align}
\op(16n+10)&\equiv 0 \pmod{8},\nonumber\\
\label{ovc3}
\op(27n+18)&\equiv 0 \pmod{3}
\end{align}
for all $n\geq 0$.  We now use these congruence results for $\overline{p}(n)$ to prove divisibilities for $\modd(-2,t;N)$ since, by \eqref{exformm2}, we have
$$\modd(-2,t;N)=\sum_{k+n^2=N}\op(k)\cdot c_n(-2,t).$$

For this first group of congruences we just need the divisibility properties of 
$\op(k)$.  

\begin{theorem} For all $t\geq 1$ and all $N\geq 0$, 
\begin{align}
\label{vm2A-1}
\modd(-2,t;8N+r) &\equiv 0 \pmod{4}\quad \text{with $r\in\{3,6\}$},\\
\label{vm2A-2}
\modd(-2,t;9N+r) &\equiv 0 \pmod{4}\quad \text{with $r\in\{3,6\}$},\\
\label{vm2A-3}
\modd(-2,t;8N+7) &\equiv 0 \pmod{8}.
\end{align}
\end{theorem}

\begin{proof} We only show the proof of \eqref{vm2A-3}. The congruences \eqref{vm2A-1} and \eqref{vm2A-2} can be verified in the same way. Since $n^2 \equiv 0, 1, 4 \pmod8$, we then know that  $k+n^2=8N+7$ implies that $k\equiv 7,6,3 \pmod{8}$. For such values of $k$, by \eqref{ovc8}, $\nu_2(\op(k))\geq 3$ and therefore $\modd(-2,t;8N+7)\equiv 0\pmod8$.
\end{proof}

For the next group of congruences we have some restrictions on the parameter $t$ because the proofs depend also on the divisibility properties of the coefficients
$c_n(-2,t)$ given in Section \ref{coefficientsS}.  

\begin{theorem} For all $J\geq 0$ and all $N\geq 0$, 
\begin{align}
\label{vm2-1}
\modd(-2,2J+1;8N+r) &\equiv 0 \pmod{4}\quad \text{with $r\in\{0,4\}$},\\
\label{vm2-1b}
\modd(-2,2J;8N+2) &\equiv 0 \pmod{4},\\
\label{vm2-2}
\modd(-2,2J+1;8N+6) &\equiv 0 \pmod{8},\\
\label{vm2-2b}
\modd(-2,2J;8N+3) &\equiv 0 \pmod{8},\\
\label{vm2-2c}
\modd(-2,4J+3;8N+r) &\equiv 0 \pmod{8}\quad \text{with $r\in\{0,4\}$},\\
\label{vm2-2d}
\modd(-2,4J+2;16N+14) &\equiv 0 \pmod{8},\\
\label{vm2-3}
\modd(-2,4J;8N+7) &\equiv 0 \pmod{16},\\
\label{vm2-3a}
\modd(-2,8J+7;8N) &\equiv 0 \pmod{16},\\
\label{vm2-4}
\modd(-2,16J+15;8N) &\equiv 0 \pmod{32},\\
\label{vm2-5}
\modd(-2,32J+31;8N) &\equiv 0 \pmod{64}.
\end{align}
\end{theorem}

\begin{proof} 

\noindent For \eqref{vm2-1}, if $t=2J+1$ then, by 1) of Theorem \ref{Tcong-cm2} with $s=1$, we have that if $c_n(-2,2J+1)\not\equiv 0 \pmod4$, then $n$ is odd. Hence $n^2\equiv 1 \pmod8$ and if $k+1\equiv r \pmod8$ with $r=0,4$ then, by \eqref{ovc8}, $\op(k)\equiv 0 \pmod4$.

\noindent For \eqref{vm2-1b}, let $t=2J$. Since $n^2 \equiv 0,1,4 \pmod8$, we know that $k+n^2=8N+2$ implies that $k\equiv 2,1,6 \pmod8$. By \eqref{ovc8}, if $k\equiv 2,6 \pmod8$, then $\nu_2(\op(k))\geq 2$ and we are done, whereas if $k\equiv 1\pmod8$, then we only have 
 $\nu_2(\op(k))\geq 1$. In this case it suffices to observe that 
$n$ is odd and $t$ is even which implies that $c_n(-2,t)\equiv 2 \pmod2$.
The proofs of \eqref{vm2-2}, \eqref{vm2-2b} and \eqref{vm2-3} are similar.

\noindent For \eqref{vm2-2c}, by 1) of Theorem \ref{Tcong-cm2} with $s=2$, we have that if $c_n(-2,4J+3)\not\equiv 0 \pmod8$ then $n$ is odd. Hence $n^2\equiv 1 \pmod8$ and, therefore,
$$\modd(-2,4J+3;8N+r)\equiv\sum_{k+n^2=8N+r}\op(k)\cdot c_n(-2,4J+3)\equiv 0\pmod{8}$$
because, by \eqref{ovc8}, if $k+1\equiv 0,4 \pmod8$ then $\op(k)\equiv 0 \pmod8$. The proof of \eqref{vm2-2d} can be done similarly.

\noindent For \eqref{vm2-3a}, by 1) of Theorem \ref{Tcong-cm2} with $s=3$, we have that if $c_n(-2,8J+7)\not\equiv 0 \pmod{16}$ then $n^2\equiv 1 \pmod8$ and, therefore,
$$\modd(-2,8J+7;8N)\equiv\sum_{k+n^2=8N}\op(k)\cdot c_n(-2,8J+7)\equiv 0\pmod{16}$$
because, by \eqref{ovc8}, if $k+1\equiv 0 \pmod8$ then $\op(k)\equiv 0 \pmod{16}$.
The proofs of \eqref{vm2-4} and \eqref{vm2-5} are similar.
\end{proof}

\begin{theorem}\label{m2mod3}   For all $J\geq 0$ and all $N\geq 0$, 
\begin{align}
\label{vm2-10}
\modd(-2,27J+13;27N+25) &\equiv 0 \pmod{3},\\
\label{vm2-11}
\modd(-2,27J+26;27N+19) &\equiv 0 \pmod{3}.
\end{align}
\end{theorem}

\begin{proof} 
\noindent For \eqref{vm2-10}, by 2) of Theorem \ref{Tcong-cm2} we have that, if $c_n(-2,27J+13)\not\equiv 0 \pmod3$, then $n^2\equiv 7 \pmod{27}$ and, therefore,
$$\modd(-2,27J+13;27N+25)\equiv\sum_{k+n^2=27N+25}\op(k)\cdot c_n(-2,27J+13)\equiv 0\pmod{3}$$
because, by \eqref{ovc3}, if $k+7\equiv 25 \pmod{27}$, then $\op(k)\equiv 0 \pmod3$.

\noindent For \eqref{vm2-11}, by 3) of Theorem \ref{Tcong-cm2} we have that, if $c_n(-2,27J+26)\not\equiv 0 \pmod3$, then $n^2\equiv 1 \pmod{27}$ and, therefore,
$$\modd(-2,27J+26;27N+19)\equiv\sum_{k+n^2=27N+19}\op(k)\cdot c_n(-2,27J+26)\equiv 0\pmod{3}$$
because, by \eqref{ovc3}, if $k+1\equiv 19 \pmod{27}$, then $\op(k)\equiv 0 \pmod{3}$.
\end{proof}

\section{Congruences for $a=0$}\label{a0S}

From \eqref{exform0even}, we know that, for all $t\geq 1$ and all $N\geq 0$, $\modd(0,2t;4N+r)=0$ for $r=1,2,3$ and
$$\modd(0,2t;4N)=\modd(-2,t,N).$$
Therefore, the congruences of the previous section can be reinterpreted in this setting.

\noindent Next, by \eqref{exform0odd}, it follows that, for all $t\geq 0$ and all $N\geq 0$, $\modd(0,2t+1;4N+r)=0$ for $r=0,2,3$ and, together with \eqref{exformW},
$$\modd(0,2t+1;4N+1)=\sum_{k+n(n-1)=N}\op(k)\cdot c_n(0,t).$$

We now prove additional congruences satisfied in the case $a=0$.  

\begin{theorem}   For all $J\geq 0$ and all $N\geq 0$, 
$$
\modd(0,2J+1;36N+r) \equiv 0 \pmod{4}\quad \text{with $r\in\{21,33\}$.}
$$
\end{theorem}
\begin{proof} Let $r=21$. Since $36N+21=4(9N+5)+1$ and $n(n-1) \equiv 0, 2, 3, 6 \pmod9$, we know that $k+n(n-1)=9N+5$ implies that $k\equiv 5,3,2,8 \pmod9$.  For such values of $k$, $\op(k) \equiv 0 \pmod4$. A similar proof works also for the case $r=33$.
\end{proof}

\begin{theorem}   For all $J\geq 0$ and all $N\geq 0$, 
\begin{align}
\label{v0odd-1}
\modd(0,8J+7;16N+r) &\equiv 0 \pmod{4}\quad \text{with $r\in\{9,13\}$},\\
\label{v0odd-2}
\modd(0,16J+15;16N+13) &\equiv 0 \pmod{8},\\
\label{v0odd-3}
\modd(0,64J+63;32N+29) &\equiv 0 \pmod{16}\\
\label{v0odd-3b}
\modd(0,128J+127;32N+29) &\equiv 0 \pmod{32}.
\end{align}
\end{theorem}
\begin{proof} For \eqref{v0odd-1}, by 1) of Theorem \ref{Tcong-c0} we know that, if $c_n(0,4J+3)\not\equiv 0 \pmod4$, then $n(n-1)\equiv 0 \pmod4$ and, therefore,
$$\modd(0,8J+7;16N+4r+1)\equiv\sum_{k+n(n-1)=4N+r}\op(k)\cdot c_n(0,4J+3)\equiv 0\pmod{4}$$
because, by \eqref{ovc8}, if $k\equiv 2,3 \pmod4$, then $\op(k)\equiv 0 \pmod4$.

\noindent For \eqref{v0odd-2}, by 1) of Theorem \ref{Tcong-c0} we have that, if $c_n(0,8J+7)\not\equiv 0 \pmod8$, then $n(n-1)\equiv 0\pmod4$ and, therefore,
$$\modd(0,16J+15;16N+13)\equiv\sum_{k+n(n-1)=4N+3}\op(k)\cdot c_n(0,8J+7)\equiv 0\pmod{8}$$
because, by \eqref{ovc8}, if $k\equiv 3 \pmod4$, then $\op(k)\equiv 0\pmod8$.

\noindent For \eqref{v0odd-3}, by 1) of Theorem \ref{Tcong-c0} we have that, if $c_n(0,32J+31)\not\equiv 0 \pmod8$, then $n(n-1)\equiv 0 \pmod8$ and, therefore,
$$\modd(0,64J+63;32N+29)\equiv\sum_{k+n(n-1)=8N+7}\op(k)\cdot c_n(0,32J+31)\equiv 0\pmod{16}$$
because, by \eqref{ovc8}, if $k\equiv 7\pmod8$, then $\op(k)\equiv 0 \pmod{16}$.
The remaining congruence \eqref{v0odd-3b} can be done in a similar way.
\end{proof} 

\begin{theorem}   For all $J\geq 0$ and all $N\geq 0$, 
\begin{align}
\label{v0odd-4}
\modd(0,54J+25;108N+49) &\equiv 0 \pmod{3},\\
\label{v0odd-5}
\modd(0,54J+53;108N+73) &\equiv 0 \pmod{3}.
\end{align}
\end{theorem}
\begin{proof} For \eqref{v0odd-4}, by 3) of Theorem \ref{Tcong-c0} we have that, if $c_n(0,27J+12)\not\equiv 0 \pmod3$, then $n(n-1)\equiv 21 \pmod{27}$ and, therefore,
$$\modd(0,54J+25;108N+49)\equiv\sum_{k+n(n-1)=27N+12}\op(k)\cdot c_n(0,27J+12)\equiv 0\pmod{3}$$
because, by \eqref{ovc3}, if $k+21\equiv 12 \pmod{27}$, then $\op(k)\equiv 0 \pmod3$.

\noindent For \eqref{v0odd-5}, by 4) of Theorem \ref{Tcong-c0} we have that, if $c_n(0,27J+26)\not\equiv 0 \pmod{3}$, then $n(n-1)\equiv 0 \pmod{27}$ and, therefore,
$$\modd(0,54J+53;108N+73)\equiv\sum_{k+n(n-1)=27N+18}\op(k)\cdot c_n(0,27J+26)\equiv 0\pmod{3}$$
because, by \eqref{ovc3}, if $k\equiv 18 \pmod{27}$, then $\op(k)\equiv 0 \pmod3$.
\end{proof} 

\section{Congruences for $a=1$}\label{a1S}

In Section 4, we gathered results that can now be summarized in the following two convenient forms:
 
\noindent
If $24n+r>0$, then
\begin{equation}\label{pre1-24}
\nu_2(a(24n+r))\geq
\begin{cases}
1 &\text{if $r=0,4,10,12,13,14,20$,}\\
2 &\text{if $r=6,16,18,22$,}\\
3 &\text{if $r=9,19,21$.}
\end{cases}
\end{equation}

\noindent 
If $32n+r>0$, then
\begin{equation}\label{pre1-32}
\nu_2(a(32n+r))\geq
\begin{cases}
1 &\text{if $r=4,10,12,14,16,24,26,30$,}\\
2 &\text{if $r=6,20,22$,}\\
3 &\text{if $r=28$.}
\end{cases}
\end{equation}
By \eqref{exform1}, we also know that
$$\modd(1,t;N)=\sum_{k+n^2=N}a(k)\cdot c_n(1,t).$$

We now use the above to prove congruences satisfied by $\modd(1,t;N)$ for various values of $t$ and $N$.  

\begin{theorem} For all $J\geq 0$ and $N\geq 0$, 
\begin{align}
\label{v1-0}
\modd(1,J;24N+22) &\equiv 0 \pmod{2},\\     
\label{v1-0b}
\modd(1,2J+1;12N+7) &\equiv 0 \pmod{2},\\   
\label{v1-0c}
\modd(1,4J+3;24N+7) &\equiv 0 \pmod{4}.
\end{align}
\end{theorem}
\begin{proof} We only show the proof of \eqref{v1-0c}; the proofs of  \eqref{v1-0} and \eqref{v1-0b} follow in similar fashion.  

Note that, for any $n$, $n^2 \equiv 0, 1, 4, 9, 12, 16 \pmod{24}$, which implies that $k\equiv 7-n^2 \pmod{24}$ must satisfy $k\equiv 7,6,3,22,15 \pmod{24}$. By \eqref{pre1-24}, if $k\equiv 6,22 \pmod{24}$, then $a(k) \equiv 0 \pmod4$. In the other cases, when $k\equiv 7,3,15 \pmod{24}$, then $n$ has to be even, and by 2) of Theorem~\ref{Tcong-c1} with $s=2$, $c_n(1,4J+3) \equiv 0 \pmod4$. Thus, we may conclude that $\modd(1,4J+3;24N+7)\equiv 0 \pmod4$.
\end{proof}

\begin{theorem}   For all $J\geq 0$ and all $N\geq 0$, 
\begin{align}
\label{v1-1}
\modd(1,2J+1;8N+r) &\equiv 0 \pmod{2}\quad \text{with $r\in\{5,7\}$},\\
\label{v1-2}
\modd(1,16J+15;16N+7) &\equiv 0 \pmod{4}\\
\label{v1-2b}
\modd(1,32J+31;32N+r) &\equiv 0 \pmod{4}\quad \text{with $r\in\{21,29\}$},\\
\label{v1-2c}
\modd(1,64J+63;32N+29) &\equiv 0 \pmod{8}.
\end{align}
\end{theorem}
\begin{proof} 

\noindent For \eqref{v1-1}, by 1) of Theorem \ref{Tcong-c1} we know that, if $c_n(1,2J+1)\not\equiv 0 \pmod2$, then $n$ is odd. Therefore, $n^2\equiv 1 \pmod8$ and 
$$\modd(1,2J+1;8N+r)\equiv\sum_{k+n^2=8N+r}\op(k)\cdot c_n(1,2J+1)\equiv 0\pmod{2}$$
because, by \eqref{pre1-32}, if $k+1\equiv 5,7 \pmod8$, $a(k)\equiv 0 \pmod2$.

\noindent For \eqref{v1-2}, by 2) of Theorem \ref{Tcong-c1} with $s=2$ we know that, if $c_n(1,16J+15)\not\equiv 0 \pmod4$, then $n^2\equiv 1 \pmod8$ and, therefore,
$$\modd(1,16J+15;16N+7)\equiv\sum_{k+n^2=16N+7}\op(k)\cdot c_n(-2,16J+15)\equiv 0\pmod{4}$$
because, by \eqref{pre1-32}, if $k+1\equiv 7 \pmod{16}$, then $a(k)\equiv 0 \pmod4$.

\noindent The proofs of \eqref{v1-2b} and \eqref{v1-2c} are similar.
\end{proof}

We close by presenting the following pair of congruences modulo 3 which are a direct consequence of \eqref{easy3} and Theorem \ref{m2mod3}. They can also be proved by using \eqref{ovc3}.

\begin{theorem}   For all $J\geq 0$ and all $N\geq 0$, 
\begin{align*}
\modd(1,27J+13;27N+25) &\equiv 0 \pmod{3},\\
\modd(1,27J+26;27N+19) &\equiv 0 \pmod{3}.
\end{align*}
\end{theorem}

\end{document}